\documentclass{amsart}
\usepackage{amsmath}
\usepackage{amssymb}
\usepackage[all]{xy}
\usepackage{comment}
\usepackage{color}

\theoremstyle{plain}
\newtheorem{thm}{Theorem}[section]
\newtheorem{thm*}{Theorem}[section]

\newtheorem{prop}[thm]{Proposition}
\newtheorem{lemma}[thm]{Lemma}

\newtheorem{lemma*}{Lemma}

\theoremstyle{definition}
\newtheorem{defn}[thm]{Definition}

\newtheorem{question}[thm] {Question}

\newtheorem*{remark*}{Remark}

\newtheorem{ex}[thm]{Example}

\newtheorem{question*}{Question}

\numberwithin{equation}{thm}

\newcommand{\bM}{\mathbb M}

\newcommand{\cM}{\mathcal M}
\newcommand{\cN}{\mathcal N}

\newcommand{\bG}{\mathbb G}

\newcommand{\cO}{\mathcal O}

\newcommand{\bP}{\mathbb P}
\newcommand{\bZ}{\mathbb Z}
\newcommand{\bF}{\mathbb F}

\newcommand{\cC}{\mathcal C}

\newcommand{\cE}{\mathcal E}

\newcommand{\fg}{\mathfrak g}

\newcommand{\fm}{\mathfrak m}

\newcommand{\ul}{\underline}

\def\sl2{\operatorname{SL_{2(2)}}\nolimits}
\def\Ga2{\operatorname{\mathbb G_{\rm a(2)}}\nolimits}

\def\GL{\operatorname{GL}\nolimits}

\newcommand{\bH}{\mathbb H}
\newcommand{\bN}{\mathbb N}

\newcommand{\bU}{\mathbb U}

\date\today

\begin{document}

 \title[Filtrations and Growth of $\bG$-modules ]{Filtrations and Growth of $\bG$-modules}
 
 \author[ Eric M. Friedlander]
{Eric M. Friedlander$^{*}$} 

\address {Department of Mathematics, University of Southern California,
Los Angeles, CA 90089}
\email{ericmf@usc.edu}

\thanks{$^{*}$ partially supported by the Simons Foundation }

\subjclass[2010]{20G05, 20C20, 20G10}

\keywords{filtrations, coalgebras, representations}

\begin{abstract}
We investigate infinite dimensional modules for an affine group scheme
$\mathbb G$ of finite type over a field of positive characteristic $p$. 
For any subspace $X \subset \mathcal O(\mathbb G)$ of the coordinate algebra
of $\bG$, we consider the abelian subcategory
$Mod(\mathbb G,X) \subset Mod(\mathbb G)$ of ``$X$-comodules" and the left exact functor
$(-)_X: Mod(\mathbb G) \to Mod(\mathbb G,X)$ which is right adjoint to the inclusion functor.
We employ ``ascending converging sequences" $\{ X_i \}$ of subspaces of $\mathcal O(\mathbb G)$
to provide functorial filtrations $\{ M_{X_i }\}$ of each
$\mathbb G$-module $M$.   A $\mathbb G$-module $M$ is injective if and only if each $M_{X_i}$
is an injective $X_i$-comodule for some (or, equivalently, for all) such $\{ X_i \}$.

 We consider the explicit ascending converging sequence $ \{ \cO(\bG)_{\leq d,\phi} \}$
of finite dimensional subcoalgebras of $\cO(\bG)$ depending upon a
closed embedding $\phi: \bG \ \hookrightarrow \ GL_N$.  Of particular interest
to us are mock injective $\bG$-modules, modules whose support varieties are empty.
Restrictions of a $\bG$-module to each $\cO(\bG)_{\leq d,\phi}$ provide new invariants 
for $\bG$-modules.  For cofinite $\mathbb G$-modules $M$, we explore the 
the growth of $d \mapsto M_{\cO(\bG)_{\leq d,\phi}}$.
\end{abstract}

\maketitle


\section{Introduction}

One approach to studying a $\bG$-module $M$ for a connected affine group scheme $\bG$
over a field $k$ of characteristic $p > 0$ is to investigate the restriction of $M$ to Frobenius 
kernels $\bG_{(r)}, r > 0$ of $\bG$.   From some
points of view, the representation theory of finite group schemes such as $\bG_{(r)}$ resembles
the representation theory of finite groups and thus shares many useful properties.   The technique
of restricting $\bG$-modules to Frobenius kernels has been effective in studying irreducible modules
and standard finite dimensional modules especially when combined with the technique of reducing
representations in characteristic 0 to characteristic $p$ (see, for example, \cite{J}).   In recent years,
the consideration of support varieties for elementary abelian 
$p$-groups (see \cite{C}) has been vastly generalized to modules for finite group schemes (see, for example, 
\cite{FP2}), linear algebraic groups (see, for example, \cite{F15}), and various finite dimensional 
algebras (see, for example, \cite{BIKP}, \cite{F-N2}, \cite{NP}).   

However, some aspects of the representation theory of $\bG$ are not seen by restrictions to
$\bG_{(r)}$.   Of particular importance are proper mock injective modules, modules which are not injective
as $\bG$-modules but whose restrictions to every Frobenius kernel $\bG_{(r)}$ are injective \cite{F18}.
Using the lens of ``stable categories", we showed in \cite{F23} that localizing with 
respect to the (triangulated) category of mock injectives enables an analogue for linear algebraic groups
of stable module categories for finite group schemes. 

In \cite{F18}, we considered linear algebraic groups of exponential type and utilized
the filtration of a $\bG$-module $M$ determined by the filtration of $M$ by exponential degree. 
This suggested that studying $\bG$-modules for more general affine group schemes using 
restrictions of their coactions to ($k$-linear) subspaces of $\cO(\bG)$ (the coordinate algebra
of $\bG$) should constructively complement the more familiar technique of studying a 
$\bG$-module $M$ by considering its restriction to finite subgroup schemes.
A classic example of this approach is the consideration of the subcoalgebra
$\cO(\bM_{N,N})_d \hookrightarrow \cO(GL_N)$ consisting of polynomials of degree $d$ 
in the coordinate functions $x_{i,j} \in \cO(GL_N)$; the $k$-linear dual of $\cO(\bM_{N,N})_d$
is a classical Schur algebra (see \cite{Green}).

In Theorem \ref{thm:M-X}, we associate 
to an arbitrary subspace (i.e., an arbitrary $k$-vector subspace) $X \subset \cO(\bG)$ an abelian subcategory 
$i_X: Mod(\bG,X) \hookrightarrow Mod(\bG)$ of the abelian category of $\bG$-modules
 together with a left exact functor $(-)_X: Mod(\bG) \ \to \ Mod(\bG,X)$ which is right adjoint to $i_X$.
 Essentially by construction, $M_X$ is the largest $\bG$-submodule of $M$ whose coaction
 $\Delta_M: M \to M\otimes \cO(\bG)$ factors through $M\otimes X \subset M\otimes \cO(\bG)$.
In Proposition \ref{prop:X-injectivity}, we show for any ascending converging sequence $\{ X_n \}$
of subspaces of $\cO(\bG)$,\
$ X_0 \subset X_1 \subset \cdots X_n \subset \cdots \subset \ \bigcup_d X_n = \cO(\bG),$ \
and any $\bG$-module $M$ that
$M$  is injective if and only if each $M_{X_n}$ is injective in $Mod(\bG,X_n)$.  
In Proposition \ref{prop:exp-degree}, we utilize the ascending, converging sequence of possibly infinite dimensional 
subspaces $\{ \cE(d) \}$ of $\cO(\bG)$ in order to correct the formulation given in \cite{F18} of the filtration 
by exponential degree for $\bG$-modules for a linear algebraic group of exponential type.   This correction 
adds the condition that the filtration be of $\bG$-modules.

If a given subspace $X \subset \cO(\bG)$ is the underlying subspace of a subcoalgebra $C$,
then $Mod(\bG,X) \ \subset \ Mod(\bG)$ is naturally identified with $CoMod(C)$, the abelian category
of comodules for $C$.
For a given subspace $X \subset \cO(\bG)$, we consider the smallest subcoalgebra 
$\cO(\bG)_{\langle X \rangle} \subset \cO(\bG)$
containing $X$; if $X$ is finite dimensional, then so is $\cO(\bG)_{\langle X \rangle}$.   Given an ascending 
converging sequence $\{ X_n \}$ of subspaces of $\cO(\bG)$ with associated sequence
$\{\cO(\bG)_{\langle X \rangle} \}$ of subcoalgebras, the filtrations on a given $\bG$-module 
by $\{ X_n \}, \ \{\cO(\bG)_{\langle X \rangle} \}$
are related by Proposition \ref{prop:relate}.
In Definitiion \ref{defn:filt-GL}, we provide another construction of ascending converging 
finite dimensional subcoalgebras $\{ \cO(\bG)_{\leq d,\phi} \}$ of $\cO(\bG)$ 
by first explicitly defining $\{ \cO(GL_N)_{\leq d} \}$ and then using a closed embedding 
$\phi:\bG \hookrightarrow GL_N$.  As seen in Proposition \ref{prop:diff-emb}, different closed 
embedding determine ``comparable" filtrations of $\bG$-modules.

In the remainder of this paper, we investigate two classes of $\bG$-modules: the class of 
mock injective $\bG$-modules and the class of cofinite
$\bG$-modules.  Examples suggests that studying a $\bG$-module $M$
by investigating the growth of the submodules $M_{X_n}$  provides new invariants for $Mod(\bG)$.
As indicated in Section \ref{sec:questions}, many concrete questions remain to be answered.

Throughout this paper, the ground field $k$ is assumed to be of characteristic $p > 0$ for 
some prime $p$.  We use $\cO(\bG)$ to denote the coordinate algebra of $\bG$, and $k[\bG]^R$
(respectively, $k[\bG]^L$) the underlying vector space of $\cO(\bG)$ with the right (resp., left) regular representation.  For us, an affine group scheme is an affine group scheme over $k$ which is of
finite type over $k$; a linear algebraic group is a smooth and connected affine group scheme.  
A closed embedding $\bH \hookrightarrow \bG$ of affine group schemes will always mean an closed
immersion which is a morphism of group schemes.

We thank Bob Guralnick, Julia Pevtsova, Paul Sobaje, and especially Cris Negron for various helpful insights. 

\vskip .2in


\section{Filtrations by subspaces $X \subset \cO(\bG)$}
\label{sec:X-comodules}

For an affine group scheme $\bG$, we denote by $Mod(\bG)$ the
abelian category of $\bG$-modules; more precisely, the abelian category of 
comodules for $\cO(\bG)$ as a coalgebra over $k$.  Thus, $M \in Mod(\bG)$
is a vector space over $k$ equipped with a right coaction $\Delta_M: M \ \to \  M\otimes \cO(\bG)$:
a $k$-linear map determining natural (with respect to commutative $k$-algebras
$A$) $A$-linear group actions $\bG(A) \times (A \otimes M) \to (A\otimes M)$.  

\vskip .1in
\begin{defn}
\label{defn:X-comodule}
Let $i_X: X \subset \cO(\bG)$ be a subspace.
We define \ $Mod(\bG,X)$ \ to be the full subcategory of $Mod(\bG)$ whose objects
are those $\bG$-modules $M$ whose coaction $\Delta_M$ factors as
$(id_M \otimes i_X) \circ \Delta_{M,X}: M \to M\otimes X \to M \otimes \cO(\bG)$.

We refer to such $\bG$-modules as ``$X$-comodules".
\end{defn}

\vskip .1in

We utilize the following lemma investigating the closure properties of
$Mod(\bG,X) \ \subset Mod(\bG)$.

\begin{lemma}
\label{lem:ker-coker}
Let $X \subset \cO(\bG)$ be a subspace and let $M$ be an $X$-comodule.
\begin{enumerate}
\item
If $j: N \ \to \ M$ is an injective map of $\bG$-modules, then $N$ is also an $X$-comodule.
\item
If $q: M \to Q$ is a surjective map of $\bG$-modules, then  $Q$ is also an $X$-comodule.
\item
If $f: M \to N$ is a map of $\bG$-modules with $N$ an $X$-comodule, then 
the kernel and cokernel of $f$ are $X$-comodules.
\end{enumerate}
\end{lemma}

\begin{proof}
To prove (1), choose a basis $\{ m_\beta, \beta \in I \}$ of $M$ such 
that a subset of this basis is a basis for $N$, and choose a basis $\{ f_\alpha, \alpha \in A \}$ 
of $\cO(\bG)$ such that a subset of this basis is a basis for $X$.
If $m \in M$ is an element of $N$, then $\Delta(m) = \sum_{\beta,j} a_{\beta,\alpha} m_\beta \otimes f_\alpha$
lies both in $N\otimes k[\bG]$ so that $a_{\beta,\alpha} = 0$  unless $m_\beta \in N$ and lies
in $M \otimes X$ so that $a_{\beta,\alpha} = 0$   unless $f_\alpha \in X$. Thus $\Delta(m) \in N\otimes X$
if $m \in N$.

To prove assertion (2), let $j: K \to M$ be the kernel of the surjective map $q: M \to Q$.
Using  assertion (1), we have
the commutative diagram 
\begin{equation}
\label{eqn:q}
\xymatrix{
K \ar[d]_j \ar[r]^-{\Delta_{K,X}} & K \otimes X\ar[r]^-{id_K \otimes i_X}   \ar[d]^{j\otimes id} & 
K \otimes \cO(\bG) \ar[d]^{j\otimes id_{\cO(\bG]}} \\
M \ar[d]_q \ar[r]^-{\Delta_{M,X}} & M \otimes X \ar[r]^-{id_M \otimes i_X}   \ar[d]^{q\otimes id_X} &
 M \otimes \cO(\bG) \ar[d]^{q\otimes id_{\cO(\bG)}} \\
Q  \ar@/_/[rr]_{\Delta_Q} & Q \otimes X   \ar[r]^-{id_Q \otimes i_X} & Q \otimes \cO(\bG).
} 
\end{equation} 
A simple diagram chase for (\ref{eqn:q}) implies that  
$\Delta_Q$ factors uniquely through $Q\otimes X$.

To prove (3), observe that the kernel  $ker\{ f \} \subset M $ is an $X$-comodule by (1) 
and that the quotient $N \twoheadrightarrow coker\{ f \}$ is an $X$-comodule by (2).
\end{proof}

\vskip .1in

We recall that the sum $M_1+M_2\  \subset \ M$ of $\bG$-submodules $M_1, M_2$ of $M$
is the image of $M_1 \oplus M_2 \to M$.

\begin{prop}
\label{prop:colimit}
Let $M$ be a $\bG$-module and $X \subset \cO(\bG)$ be a subspace.  
If $M_1 \subset M, \ M_2 \subset M$ are 
$\bG$-submodules which are $X$-comodules, then $M_1 + M_2 \subset M$
is also an $X$-comodule.
Thus, the category $\chi(M)$ whose objects are $X$-comodules of $M$
and whose maps are inclusions of $\bG$-submodules of $M$ is a filtering subcategory
of $Mod(\bG)$.

Consequently,
\begin{equation}
\label{eqn:bracket-X}
M_X\ \equiv \ \varinjlim_{N \in \chi(M)} N \quad = \ \bigcup_{N \in \chi(M)} N \ \subset \ M
\end{equation}
is well defined as a $\bG$-submodule.  Moreover, \ $M_X  \subset M$ \ is an $X$-comodule,
the largest $X$-comodule contained in $M$.
\end{prop}

\begin{proof}
Recall that $M_1 + M_2 \subset M$ fits in a short exact sequence of $\bG$-modules
\ $0 \to \ M_1\cap \ M_2 \ \to \ M_1\oplus M_2 \ \to \ M_1+M_2 \to 0$.
Since $M_1 \oplus M_2$ is clearly an $X$-comodule whenever $M_1, M_2$ are $X$-comodules,
the first assertion follows from Lemma \ref{lem:ker-coker}.

This implies that the category $\chi(M)$ is filtering;  given two objects $N_1 \subset M$ and $N_2 \subset M$
of $\chi(M)$, both map to $N_1+N_2 \subset M$ which is an object
of $\chi(M)$.   Thus, $\varinjlim_{N \in \chi(M)} N \ \to M$  
equals  the union \ $\bigcup_{N \in \chi(M)} N  \ \subset \ M$.    Recall that $(-) \otimes V$ for a 
given vector space $V$ commutes with filtered colimits of $k$-vector spaces.  Consequently, 
$$\varinjlim_{N \in \chi(M)} \Delta_N: \varinjlim_{N \in \chi(M)} N \ \to \ 
\varinjlim_{N \in \chi(M)} (N \otimes \cO(\bG]) \ = \ (\varinjlim_{N \in \chi(M)} N) \otimes \cO(\bG)$$ 
factors through 
$ \varinjlim_{N \in \chi(M)} N \to (\varinjlim_{N \in \chi(M)} N) \otimes X$.  In other wordss,
 $M_X$ is an $X$-comodule, the largest $X$-comodule contained in $M$.
\end{proof}

\vskip .1in

Theorem \ref{thm:M-X} introduces the functor $(-)_X: Mod(\bG) \quad \to \quad Mod(\bG,X)$
right adjoint to the natural embedding.

\begin{thm}
\label{thm:M-X}
Let $\bG$ be an affine group scheme of finite type over $k$ and let $i_X: X \subset \cO(\bG)$ be a subspace. 
Denote by  
$$i_{X*}: Mod(\bG,X) \quad \hookrightarrow \quad Mod(\bG)$$
the full subcategory of $Mod(\bG)$ whose objects are $X$-comodules.
\begin{enumerate}
\item
$Mod(\bG,X)$ is an abelian subcategory which is closed under filtering colimits.
\item
Sending a $\bG$-module $M$
to the $\bG$-submodule $M_X$ of $M$ as in (\ref{eqn:bracket-X}) determines a functor
$$(-)_X: Mod(\bG) \quad \to \quad Mod(\bG,X).$$ 
\item
$(-)_X$ is left exact and is right adjoint to the embedding functor $i_{X*}: Mod(\bG,X) \to Mod(\bG)$.
\end{enumerate}
\end{thm}

\begin{proof}
The fact that $Mod(\bG,X)$ is an abelian subcategory of $Mod(\bG)$ follows directly from
Lemma \ref{lem:ker-coker}.    To prove that  $Mod(\bG,X)$ is closed under colimits indexed
by a filtering category $I$, observe that the natural map $\varinjlim_i (M_i \otimes X) \ \to \
\varinjlim_i (M_i) \otimes X$ is an isomorphism.  Thus, if each $M_i$ is an $X$-comodule, 
so is $\varinjlim_i (M_i)$.

To prove functoriality of $(-)_X$, observe that if $f: M \to N$ is a map in $Mod(\bG)$ 
then $f(M_X) \ \subset \ N$ is
contained in $N_X$  by Lemma \ref{lem:ker-coker}(2) and the equality 
$N_X   \ = \ \bigcup_{N^\prime \in \chi_M} N^\prime \ \subset N$ of (\ref{eqn:bracket-X}).  
This equality also shows that $(-)_X$ is left exact.

Functoriality together with (1.3.1) determines the natural inclusion
$$Hom_{Mod(\bG)}(i_{X*}(M),N) \quad \hookrightarrow \quad Hom_{Mod(\bG,X)}(M,N_X)$$
inverse to the  inclusion $Hom_{Mod(\bG,X)}(M,N_X) \ \hookrightarrow \ Hom_{Mod(\bG)}(i_{X*}(M),N)$
and thus a bijection.  This is the required isomorphism for the asserted adjunction.
\end{proof}

\vskip .1in

Perhaps the simplest example of the functor $(-)_X: Mod(\bG) \quad \to \quad Mod(\bG,X)$
is the case in which $X = k$, the span of $1 \in \cO(\bG)$.  In this case,
$(-)_X \ = \ H^0(\bG,-)$.  Observe that $H^0(\bG,-)$ is left exact for any $\bG$, but is not always exact.

\vskip .1in

We say a sequence of subspaces \ $\{ X_i \}$ of $\cO(\bG)$ indexed by the non-negative
numbers $i \geq 0$ is an {\it ascending 
converging sequence of subspaces} of $\cO(\bG)$
if $X_i \subset X_{i+1}$ for all $i \geq 0$ and if \
 $\bigcup_{i \geq 0} X_i \ = \  \cO(\bG)$.  

\vskip .1in

\begin{prop}
\label{prop:exhaust-X}
Let $\bG$ be an affine group scheme of finite type over $k$ and let $\{ X_i \}$
be an ascending converging sequence of subspaces of $\cO(\bG)$.
If $M$ is a finite dimensional $\bG$-module, then $M$ is an $X_i$-comodule for all $i \gg 0.$

Sending a $\bG$-module $M$ to the sequence of $\bG$-submodules
\begin{equation}
\label{eqn:X-seq}
M_{X_0} \ \subset \ M_{X_1} \ \subset M_{X_2} \subset \cdots \ \subset \ \bigcup_{i \geq 0} M_{X_i} \ = \ M
\end{equation}
is a filtration of $M$, functorial in $M$,  with the property that each $M_{X_i}$ is an $X_i$-comodule.

\end{prop}

\begin{proof}
If the $\bG$-module $M$ is finite dimensional, then $\Delta_M: M \to M\otimes \cO(\bG)$ must have
image in some finite dimensional subspace of $M \otimes X$ and
thus must have image contained in some $M\otimes X_i$.   If $M$ 
is an arbitrary $\bG$-module, then $M$
is locally finite so that every $m \in M$ lies in some finite dimensional
$\bG$-submodule $M^\prime \subset M$  and thus must be contained in some $M_{X_i}$ as required.
\end{proof}

\vskip .1in

We argue exactly as in the proof of \cite[Prop 4.2]{F18} to conclude the following
detection of rational injectivity of a $\bG$-module.
 
\begin{prop}
\label{prop:X-injectivity}
Consider an affine group scheme $\bG$ of finite type over $k$ and choose an 
ascending converging sequence $\{ X_i\}$ of subspaces of $\cO(\bG)$.
Then a $\bG$-module $L$ is injective  if and only if $L_{X_i}$ is
an injective object of $Mod(\bG,X_i)$ for all $i \geq 0$.
\end{prop}

\begin{proof}
If $L$ is injective, then the adjunction of the left exactness of $(-)_X$ and the exactness 
of its left adjoint $i_{X*}(-)$ tell us that $L_{X_i} \subset L$ 
is an injective object of $Mod(\bG,X_i)$ for each $i \geq 0$.

Assume now that the $\bG$-module $L$ has the property that 
$L_{X_i} \subset L$ is an injective object of $Mod(\bG,X_i)$ for all $i \geq 0$.  
Let $M^\prime \hookrightarrow M$ be an inclusion of $G$-modules and consider
a map  $f^\prime: M^\prime \to L$ of $\bG$-modules.  We inductively 
construct an extension of $f: M \to L$ of $f^\prime$ as follows.
Denote by $f^\prime_d: M^\prime_{X_d} \to L_{X_d}$ the restriction of
$f^\prime$ to $M^\prime_{X_d}$.  Choose $f_d: M_{X_d} \to L_{X_d}$ 
extending $f_d^\prime + f_{d-1}: (M^\prime)_{X_d} + M_{X_{d-1}} \to L_{X_d}$ using the 
injectivity of $L_{X_d}$ as an object of $Mod(\bG,X_d)$ (and taking $M_{X_{-1}} = 0$).  We define
$f: M \to L$ extending $f^\prime$ to be $\varinjlim_i f_i: M = \varinjlim_i ((M^\prime)_{X_i} + M_{X_{i-1}})  \to L$.
\end{proof}

\vskip .1in

We proceed to give in Proposition \ref{prop:exp-degree} a simple ``fix" for the 
``filtration by exponential degree" of a $\bG$-module $M$ for a linear algebraic group $\bG$ 
of exponential type given in \cite{F18}.  Our modification provides the largest filtration 
subordinate to that of \cite{F18} which is a filtration by $\bG$-submodules. 

We recall the definition of a linear algebraic group $\bG$ of exponential type, a class of affine 
algebraic groups for which there is a somewhat explicit geometric description of the support
varieties of its representations.  Let $\cN_p(\fg)$ denote the $p$-nilpotent 
cone of the Lie algebra $\fg = Lie(\bG)$.  Thus, $\cN_p(\fg) \ \subset \ \fg$ consists
of those $X \in \fg$ such that $X^{[p]} = 0$.  We utilize the notation $\cC_r(N_p(\fg))$
to denote the commuting variety of $r$-tuples of pair-wise commuting, $p$-nilpotent elements of $\fg$.
We recall from \cite[Thm1.5]{SFB1} the scheme $V(\bG_{(r)})$ representing the functor of 1-parameter 
subgroups of the infinitesimal group scheme $\bG_{(r)}$.

\begin{defn}
\label{defn:exp-type}
\cite[Defn1.6]{F15}
Let $\bG$ be a linear algebraic group with Lie algebra $\fg$.   A structure of exponential type
on $\bG$ is a $\bG$-equivariant morphism of $k$-schemes (with respect to adjoint actions)
\begin{equation}
\label{eqn:Exp}
\cE: \cN_p(\fg) \times \bG_a \ \to \bG, \quad (B,s) \mapsto \cE_B(s)
\end{equation}
satisfying the following conditions for all field extensions $K/k$:
\begin{enumerate}
\item
For each $B\in \cN_p(\fg)(K)$, $\cE_B: \bG_{a,K} \to \bG_K$ is a 1-parameter subgroup.
\item
For any pair of  commuting $p$-nilpotent elements $B, B^\prime \in \fg_K$,
the maps $\cE_B, \cE_{B^\prime}: \bG_{a,K} \to \bG_K$ commute.
\item
For any $\alpha \in K$,  and any 
$s \in \bG_a(K)$, \ $\cE_{\alpha \cdot B}(s) = \cE_B(\alpha\cdot s)$.
\item  Every 1-parameter subgroup $\psi: \bG_{a,K} \to \bG_K$ is of the form 
$$ \cE_{\ul B} \ \equiv \ \prod_{s=0}^{r-1} (\cE_{B_s} \circ F^s)$$
for some $r > 0$, some $\ul B \in \cC_r(\cN_p(\fg_K))$.
\item   The natural
map $\cC_r(\cN_p(\fg)) \to V(\bG_{(r)})$ induces a bijection on $K$-points
sending $\ul B$ to the infinitesimal 1-parameter subgroup $\bG_{a(r),K} \to \bG_{(r),K}$
which factors $\cE_{\ul B} \circ i_r: \bG_{a(r),K} \to \bG_{a,K} \to \bG_{K}$.
\end{enumerate}
\end{defn}

\vskip .1in

Various examples of $\bG$ of exponential type are considered in \cite{Sobj13}; these
include simple classical groups, their standard parabolic subgroups, and the
unipotent radicals of these parabolic subgroups.

\begin{prop} (\cite[Defn 4.5]{F15})
\label{prop:exp-degree}
Let $(\bG,\cE)$ be a linear algebraic group of exponential type.  We define $\cE(\bG)_d
\ \hookrightarrow \ \cO(\bG)$ to be the subspace 
\begin{equation}
\label{eqn:E(d)}
\cE(\bG)_d \ \equiv \ \cE^{*-1}(k[(\cN_p(\fg)][t]_{\leq d}) \quad \subset \quad \cO(\bG)
\end{equation}
where $k[\cN_p(\fg)][t]_{\leq d} \subset k[\cN_p(\fg) \times \bG_a]$ is the 
subspace of polynomials in $k[\cN_p(\fg)][t]$ of degree $\leq d$.  

So defined, 
$\{ M_{\cE(\bG)_d}\}$ is the coarsest filtration of $M$ by $\bG$-modules subordinate
to the `` filtration by exponential degree" of \cite{F18}.
\end{prop}

\begin{proof}
The ``filtration by exponential degree" of \cite[Defn 3.10]{F18} associates to the $\bG$-module $M$
and a positive integer $d$ the subspace $M_{[d]} \subset M$ consisting of elements 
$m \in M$ with the property that $\Delta_M(m) \in M\otimes \cE(\bG)_d$.   By Proposition \ref{prop:colimit},
$M_{\cE(\bG)_d}$ is the largest $\bG$-submodule of $M$ such that $M_{\cE(\bG)_d} \subset M_{[d]}$.
\end{proof}

\vskip .1in

In the following proposition, $\cE_B: \bG_{a,K} \to \bG_K$ is the exponential map determined
by a $K$-point $B$ of $\cN_p(\fg)$ (for some field extension $K/k$)
and the exponential structure $\cE: \cN_p(\bG) \times \bG_a \to \bG$.
For any $s \geq 0$, $u_s: k[t] \to k$ is the $k$-linear map sending $t^i$ to 0 if $i\not= p^s$
and sending $t^{p^s}$ to 1.  We denote by $(\cE_B)_*(u_s): \cO(\bG_K) \to k$  the linear map given
by the composition $u_s \circ (\cE_B)^*: \cO(\bG_K) \to K[t] \to K$.

We utilize the (``$\pi$-point") support variety $M \mapsto \Pi(\bG)_M$ of \cite{F23} extending the construction 
for finite group schemes given in \cite{FP2}.

We justify saying that $\Pi(\bG)_M$ is the 
``inverse image under the projection" of $\Pi(\bG_{(r)})_{M|\bG_{(r)}}$ by recalling that
$\Pi(\bG)$ (respectively, $\Pi(\bG_r)$) for $\bG$ of exponential type can be identified with 
$\bP(\cC_\infty(\bG))$ (resp., $\bP(\cC_r(\bG))$) and there is a natural projection
$\cC_\infty(\bG) \to \cC_r(\bG)$.

\begin{prop}
\label{prop:support-exp}
Let $(\bG,\cE)$ be a linear algebraic group of exponential type and let $M$ be a $\bG$-module 
with the property that the coaction $\Delta_M: M \to M\otimes \cO(\bG)$ factors through
$M\otimes \cE(\bG)_{p^r} \hookrightarrow M \otimes \cO(\bG)$; in other words, assume 
that $M = M_{\cE(\bG)_{p^r}}$.   Then, for any
$K$-point $B$ of $\cN_p(\fg)$, $(\cE_B)_*(u_s)$ acts trivially on $M_K$ provided that $s \geq r$.

Consequently, if  $M= M_{\cE(\bG)_{p^r}}$, then the support variety $\Pi(\bG)_M$ of $M$ is the 
``inverse image under the projection" of $\Pi(\bG_{(r)})_{M|\bG_{(r)}}$ (containing the
center of the ``projection" $\Pi(\bG) \to \Pi(\bG_{(r)})$, where $M|\bG_{(r)}$ denotes the 
restriction of $M$ to the Frobenius kernel $\bG_{(r)} \hookrightarrow \bG$.
\end{prop}

\begin{proof}
The proposition follows immediately from Proposition \ref{prop:exp-degree} and
\cite[Prop 3.17]{F18}.
\end{proof}

\vskip .2in


\section{Ascending, converging sequences of subcoalgebras}
\label{sec:coalgebras}

In this section, we investigate ascending, converging sequences $\{ C_i \}$ of subcoalgebras
of $\bG$.  The advantage of considering a subcoalgebra $C \subset  \cO(\bG)$ rather
than an arbitrary subspace $X \subset \cO(\bG)$ is that the abelian subcategory $Mod(\bG,C)$
as in Definition \ref{defn:X-comodule} (with $X = C$) is equal to the category of comodules
for $C$,
\begin{equation}
\label{eqn:advantage}
 Mod(\bG,C) \ = \ CoMod(C) \ \subset \ Mod(\bG).
 \end{equation}
Moreover,  if $C \subset \cO(\bG)$ is a subcoalgebra and $M$ is a $\bG$-module, then 
 \begin{equation}
\label{eqn:advantage2}
M_C \quad = \quad \{ m \in M: \Delta_M(m) \in M \otimes C \} \ \subset \ M.
\end{equation}
A useful consequence of (\ref{eqn:advantage2}) is the equality
\begin{equation}
\label{eqn:advantage3}
(k[\bG]^R)_C  \quad = \quad C
\end{equation}
provided that $C$ contains the unit of $\cO(\bG)$.  This
follows from the observation that $\Delta: \cO(\bG) \to \cO(\bG)$
satisfies the condition that $\Delta(f) - (f\otimes 1 + 1\otimes f)$ lies
in $I \otimes I$ for all $f \in I$, where $I \subset \cO(\bG)$ is the augmentation
ideal of $\cO(\bG)$ (see [I.2.4]{J}) 

\vskip .1in

The following theorem, called the ``Fundamental Theorem of 
Coalgebras" in \cite{Sw} and the ``Finiteness Theorem" in \cite{smontgom} 
(when stated for arbitrary coalgebras) has the following appealing form
when specialized to the Hopf algebra $\cO(\bG)$.  We use the notation of \cite[I.2.13]{J}:
$k\bG\cdot X$ for the $\bG$-module submodule of the right regular representation
of $\bG$ on $\cO(\bG)$ (given by $f(-) \mapsto f(-g)$) generated by elements of $X$; $k\bG^{op} \cdot Y$ is the 
$\bG^{op}$-submodule of the right regular representation of $\bG^{op}$ on $\cO(\bG^{op})$
(given by $(g,f(-)) \mapsto f(g-)$) generated by elements of $Y$.

\begin{thm}
\label{thm:finite}
Let $\bG$ be an affine group scheme and $X \subset \cO(\bG)$ be a subspace.
Then there is a smallest subcoalgebra of $\cO(\bG)$ containing $X$, 
$\cO(\bG)_{\langle X \rangle} \subset \cO(\bG)$, given by 
$$\cO(\bG)_{\langle X \rangle}  \quad \equiv \quad k\bG^{op} \cdot (k\bG \cdot X).$$
If $X$ is finite dimensional, then  $\cO(\bG)_{\langle X \rangle} $ is also finite dimensional.

Consequently, for any ascending, converging sequences $\{ X_i \}$
of finite dimensional subspaces of $\cO(\bG)$,  there is a smallest ascending, 
converging sequence $\{ \cO(\bG)_{\langle X_i \rangle}  \}$ of finite dimensional subcoalgebras 
of $\cO(\bG)$ satisfying the condition that  $X_i \subset  \cO(\bG)_{\langle X_i \rangle}$.
\end{thm}

\begin{proof}
The proof that $k\bG\cdot X$ is the $\bG$-submodule of $k[G]^R$ generated by $X$
is given in \cite[I.2.13]{J}, implicitly proving that $k\bG\cdot X$ is finite dimensional
if $X$ is finite dimensional.  Applying this to $\bG^{op}$ and the subspace $k\bG\cdot X \subset
k[\bG^{op}]^R$, we conclude that \ $\cO(\bG)_{\langle X \rangle}  \ \equiv \ k\bG^{op} \cdot (k\bG \cdot X)$ \
is finite dimensional if $X$ is finite dimensional and is a $\cO(\bG^{op})$-subcomodule of
$k[\bG^{op}]^R$.  Because the right regular actions of $\bG$ of $\cO(\bG)$ and $\bG^{op}$
on $\cO(\bG^{op})$ commute once one identifies $\cO(\bG)$ with $\cO(\bG^{op})$ as
$k$-vector spaces, we conclude that $\cO(\bG)_{\langle X \rangle} $ is also a $\cO(\bG)$-subcomodule of
$k[\bG]^R$.  This implies that $\cO(\bG)_{\langle X \rangle}  \subset \cO(\bG)$ is a subcoalgebra.
\end{proof}

\vskip .1in

In the following proposition, we relate the subcategories 
$Mod(\bG,X)$ and \\ $CoMod(\cO(\bG)_{\langle X \rangle})$ of $Mod(\bG)$.  

\begin{prop}
\label{prop:relate}
Consider an affine group scheme $\bG$ and a family $\cM \subset Mod(\bG)$ of $\bG$-modules.
Let $X \subset \cO(\bG)$ be the smallest subspace such that $M \in Mod(\bG,X)$ for all $M \in \cM$.
Then $\cO(\bG)_{\langle X \rangle} $ is the smallest subcoalgebra $C \subset \cO(\bG)$ with the property that
$Mod(\bG,X) \ \hookrightarrow \  CoMod(C) \subset Mod(\bG)$. 

In particular, $Mod(\bG,\cO(\bG)_{\langle X \rangle}) \ = \ CoMod(\cO(\bG)_{\langle X \rangle} )$.
\end{prop}

\begin{proof} 
Write $\cM  = \{ M_\alpha \}$.  Choose a basis $\{m_{\alpha,i_\alpha} \}$ for each $M_\alpha$ and
write $\Delta_{M_\alpha}(m_{\alpha,j_\alpha}) = \sum m_{\alpha,(i,j)_\alpha} \otimes f_{\alpha,(i,j)_\alpha}$.
Then $X$ is the span of $\{ f_{\alpha,(i,j)_\alpha} \}$.   Thus, $\Delta_{M_\alpha} \subset 
M_\alpha \otimes \cO(\bG)_X$ for every $M_\alpha$ in $\cM$, so that $X \subset \cO(\bG)_{\langle X \rangle} $.

On the other hand, if $C\subset \cO(\bG)$ is a subcolagebra with the property that every
$M_\alpha \in \cM$ is a $C$-comodule, then each $f_{\alpha,(i,j)_\alpha}$ must be an element
of $C$.
\end{proof}
%
%
%

\vskip .1in

The tensor product of $\cO(\bG)$-comodules (i.e., $\bG$-modules) involves the product
structure $\cO(\bG) \otimes \cO(\bG) \to \cO(\bG)$ induced by the diagonal $diag: 
\bG \hookrightarrow \bG \times \bG$.  Thus, unless the subcoalgebra $C \subset \cO(\bG)$
is also a subalgebra, this tensor product does not induce a tensor product structure on $CoMod(C)$.  

The following suggests a useful condition on ascending converging sequences of subcoalgebras
of $\cO(\bG)$.

\begin{prop}
\label{prop:prod-coalg}
Let $\bG$ be an affine group scheme.  Let $X, \ Y \ \subset \cO(\bG)$ be subspaces
and set $X \cdot Y \subset \cO(\bG)$ be the subspace spanned by products $x \cdot y$
with $x \in X, \ y \in Y$.  For any $\bG$-modules $M$ and $N$, the $\cO(\bG)$-module
$M_X \otimes M_Y$ is contained in $(M\otimes N)_{X\cdot Y}$.  Consequently,
\begin{enumerate}
\item
For subcoalgebras $C, \ C^\prime, \ C^{\prime\prime}$ such that the 
multiplication map $\mu: \cO(\bG) \otimes \cO(\bG) \to \cO(\bG) $ restricts to
$C \otimes C^\prime \to C^{\prime\prime}$, there is a natural map of $\bG$-modules
\ $M_C \otimes N_{C^\prime} \ \to \ (M\otimes N)_{C^{\prime\prime}}$
for every pair of $\bG$-modules $M, N$.
\item
If $(\bG,\cE)$ is a linear algebraic group of exponential type, there os a natural map
of $\bG$-modules  \ $M_{\cE(d)} \otimes N_{\cE(e)} \ \to \ (M\otimes N)_{\cE(d+e)}$
for every pair of $\bG$-modules $M, N$.
\end{enumerate}
\end{prop}

\begin{proof}  The proof follows easily from the observation 
the coaction map $\Delta_{M\otimes N}: M\otimes N \ \to \ (M\otimes N) \otimes \cO(\bG)$  arises by 
composing $\Delta_M \otimes \Delta_N$ 
with the product map $\cO(\bG) \otimes \cO(\bG) \ \to \ \cO(\bG)$.
\end{proof}

\vskip .1in

Any affine group scheme $\bG$ is a closed subgroup scheme of some $GL_N$.
Given a closed emebedding $\phi: \bG  \hookrightarrow GL_N$ of group schemes 
with map on coordinate algebras $\phi^*: \cO_{GL_N} \twoheadrightarrow \cO(\bG)$ and an ascending,
converging sequence $\{ C_i \subset \cO(GL_N) \}$ of subcoalgebras of $\cO(GL_N)$, restricting 
along $\phi^*$ determines the ascending, converging sequence of subcoalgebras
$\{ \phi^*(C_i) \subset \cO(\bG) \}$.   This motivates our formulation in Proposition \ref{prop:O(GLN)} of an explicit  
ascending, converging sequence $\{ \cO(GL_N)_{\leq d} \}$ of subcoalgebras of $\cO(GL_N)$.

We will implicitly use the following observation:  for any surjective map $\phi: C \twoheadrightarrow C^\prime$
of coalgebras and any subcoalgebra $D \hookrightarrow C$, $D^\prime \equiv \phi(D) \subset C^\prime$ 
is a subcoalgebra.  This is easily verified using the fact that the map $\Phi$ of coalgebras commutes with 
coproducts, so that the coproduct on $D$ induces a coproduct of $D^\prime$ and the counit of $C^\prime$
restricts to a counit for $D^\prime$.

\begin{defn}
\label{defn:filt-GL}
Let $\cO(\bM_{N, N})$ be the bialgebra given as the coordinate algebra of the affine 
variety of $N \times N$ matrices with monoid structure given by matrix multiplication.
For any $r \geq 0$, we define the subspace $\cO(\bM_{N, N})_{\leq r}$ to
be the subspace of the polynomial algebra in the $N^2$-variables $x_{i,j}$ consisting of
polynomials of total degree  $\leq r$.   Let $\cO(\bG_m)$ be the bialgebra with 
coordinate algebra $k[t,t^{-1}]$ with respect to whose coproduct both $t$ and $t^{-1}$
are primitive,
and define $\cO(\bG_m)_{\leq s}$ to be the span of $\{ t^i, \ N\cdot |i| \leq s \}$.

Consider the closed immersion of monoid schemes
\begin{equation}
\label{eqn:embed} 
\eta: GL_N \ \hookrightarrow \ \bM_{N,N} \times \bG_m, \quad A \mapsto (A,det(A)^{-1}),
\end{equation}
identifying $GL_N$ as the zero locus of the function 
$$(\sum_{\sigma\in \Sigma_N} (-1)^{sgn(\sigma)}\prod_{1\leq i\leq N} x_{i,\sigma(i)}) \ \otimes \ t^{-1}.$$
We give $\cO(\bM_{N,N} \times \bG_m)$ the filtration defined by the tensor product of the above filtrations on
$\cO(\bM_{N, N})$ and $\cO(\bG_m)$.

For any $d \geq 0$, define
$\cO(GL_N)_{\leq d} \ \equiv \ \eta^*((\cO(\bM_{N,N} \times \bG_m)_{\leq d})$.
\end{defn}

\vskip .1in

\begin{prop}
\label{prop:O(GLN)}
Adopt the notation of Definition \ref{defn:filt-GL}.
\begin{enumerate}
\item
$\eta^*: \cO(\bG_m) \otimes \cO(\bM_{N,N}) \ \twoheadrightarrow \ \cO(GL_N)$
is a surjective map of filtered coalgebras.
\item
For any $d \geq 0$, $\cO(GL_N)_{\leq d} \ \hookrightarrow \  \cO(GL_N)$ is a subcoalgebra.
\item
$\{ \cO(GL_N)_{\leq d} \}$ is an ascending, converging sequence of finite 
dimensional subcoalgebras of $\cO(GL_N)$.
\item
The function sending $d$ to $dim(\cO(GL_N)_{\leq d})$ for a fixed $N$
differs from the function $d \mapsto \frac{1}{N^2!}{d^{N^2}}$ by a function 
bounded by a polynomial in $d$ of degree less that $N^2$.
\end{enumerate}
\end{prop}

\begin{proof}
To prove (1), first observe that $\eta^*$ is a map of coalgebras since 
$\eta$ is a homomorphism of monoid schemes.  The surjectivity of $\eta^*$ 
follows from the fact $\eta$ is a closed immersion.   Assertion (2) follows from
the fact that both $\cO(\bM_{N, N})_{\leq r} \hookrightarrow \cO(\bM_{N, N})_{\leq r}$
and $\cO(\bG_m)_{\leq s} \hookrightarrow \cO(\bG_m)$ are subcoalgebras so that
$(\cO(\bM_{N,N} \times \bG_m))_{\leq d} \hookrightarrow \cO(\bM_{N,N} \times \bG_m)$
is also a subcoalgebra.  Assertion (3) follows immediately.

We easily verify that 
\begin{equation}
\label{eqn:dim}
dim(\cO(\bM_{N\times N})_{\leq r}) \ = \ {r + N^2\choose N^2}, \quad dim(\cO(\bG_m)_s )
\ = \ 2s+1.
\end{equation}
and that $dim(\cO(GL_N)_{\leq r}) - dim(\cO(\bM_{N\times N})_{\leq r})$ equals 
$$\sum_{1 \leq i \leq [r/N]} dim(\cO(\bM_{N\times N})_{\leq r-iN}) - dim(\cO(\bM_{N\times N})_{\leq r-(i+1)N}).$$
Thus, the difference of the function $r \mapsto dim(\cO(GL_N)_{\leq r})$ and the 
function $r \mapsto \frac{1}{N^2!}{r^{N^2}}$ has growth less than polynomial in $r$ of 
degree less than $N^2$ as in assertion (4).
\end{proof}

\vskip .1in

\begin{defn}
\label{defn:O(bG)}
Consider an affine group scheme $\bG$ equipped with a closed embedding $\phi: \bG \hookrightarrow GL_N$
for some $N$.
We define the ascending, converging filtration $\{ \cO(\bG)_{\leq d,\phi},  d > 0 \}$ of subcoalgebras of $\cO(\bG)$
by setting $\cO(\bG))_{\leq d,\phi}$ equal to $\phi^*(\cO(GL_N)_{\leq d})$.  
\end{defn}

\vskip .1in

For various unipotent linear algebraic groups $\phi: \bU \hookrightarrow GL_N$, we give a 
familiar description of  $\{ \cO(\bU)_{\leq d,\phi} \}$.

\begin{ex} \cite[Ex 2.5]{F18}
\label{ex:unipotent}
Let $\phi: \bU_N \hookrightarrow GL_N$ be the unipotent radical of the Borel subgroup of upper
triangular matrices of $GL_N$.  Then $\phi^*: \cO(GL_N) \ \twoheadrightarrow \ \cO(\bU_N) \simeq k[y_{i,j}; i < j]$
is given by  \ $x_{i,j} \mapsto y_{i,j}, \ i < j$, $x_{i,i} \mapsto 1$, and $x_{i,j} \mapsto 0, \ i > j$.
The coproduct on $\cO(\bU_N) \simeq k[y_{i,j}; i < j]$ is given by
$$\Delta_{\bU_N}(y_{i,j}) \quad = \quad (y_{i,j}\otimes 1) + (\sum_{i < t < j} (y_i\otimes y_t + y_t \otimes y_j)) 
+ (1\otimes y_{i,j}).$$
We identify the subcoalgebra $\cO(\bU_N)_{\leq d,\phi} \subset \cO(\bU_N)$ with 
the subspace of $k[y_{i,j}; i < j]$ consisting of polynomials of total degree $\leq d$ and with coproduct 
the restriction of $\Delta_{\bU_N}$ as above.  Thus,
$$dim(\cO(\bU_N)_{\leq d,\phi}) \quad = \quad \binom{N(N-1)+d}{N(N-1)}$$
which differs from $\frac{1}{(N(N-1))!}d^{N(N-1)}$ by a function bounded by a polynomial
in $d$ of degree less than $N(N-1)$.

A similar description is given for $\cO(\bU)_{\leq d,\phi}$ for any closed subgroup scheme 
$\phi: \bU \subset \bU_N \subset GL_N$ such that
the composition $\cO(GL_N) \stackrel{j^*}{\to} \cO(\bU)$ is surjective.  
\end{ex}

\vskip .1in

We show that changing the embedding $\phi: \bG \hookrightarrow GL_N$ has limited effect
upon the associated ascending converging sequences of subcoalgebras of $\cO(\bG)$.

\begin{prop}
\label{prop:diff-emb}
Let $\bG$ be an affine group scheme  and consider two closed 
embeddings \ $\phi: \bG  \hookrightarrow GL_N, \ \phi^\prime: \bG  \hookrightarrow GL_{N^\prime}.$
There exist positive numbers $c, c^\prime$ such that 
$$\cO(\bG)_{\leq d,\phi} \ \subset  \cO(\bG)_{\leq c\cdot d, \phi^\prime}, \quad
\cO(\bG)_{\leq d,\phi^\prime} \ \subset  \cO(\bG)_{\leq c^\prime \cdot d,\phi}$$
for all $d \geq 0$.
\end{prop}

\begin{proof}
Using the closed embedding 
$$\bigoplus (\phi \times 1)\oplus (1 \times \phi^\prime): 
\bG \ \hookrightarrow \ GL_{N} \times GL_{N^\prime} \ \hookrightarrow \ GL_M, \quad
M = N + N^\prime,$$
we conclude it suffices to consider a closed embedding $\phi: \bG \hookrightarrow GL_N$ 
and compare $d \mapsto \cO(\bG)_{\leq d,\phi}$ with  $d \mapsto \cO(\bG)_{\leq d,\psi \circ \phi}$
where $\psi = \oplus \circ (id \times 1): GL_{N} \times GL_{N^\prime} \ \hookrightarrow \ GL_M$.
Since $\psi^*$ is surjective, each coordinate function $y_{s,t} \in \cO(GL_N)$ is the image
under $\psi^*$ of some $f_{s,t}(x_{i,j}) \in \cO(GL_M)$.  Denote by $c$ the maximum of the degrees of
$f_{s,t}, 1 \leq s,t \leq n$.  Then $\cO(\bG)_{\leq d,\phi} \subset \cO(\bG)_{\leq c \cdot d,\psi \circ \phi}$.

Denote by $c^\prime$ the maximum of the degrees of $\phi^*(x_{i,j}) \in \cO(GL_N)$.
Then $\cO(\bG)_{\leq d,\psi \circ \phi} \ \subset  \cO(\bG)_{\leq c^\prime\cdot d,\phi}$.

\vskip .2in 
%
%
%
\end{proof}

\vskip .2in


\section{Filtering $\bG$-modules using $\{ \cO(\bG)_{\leq d,j} \}$}
\label{sec:filt}

We state as a proposition the evident consequence of functoriality of $(-)_X$.

\begin{prop}
\label{prop:functorial}
For any ascending, converging sequence $\{ X_i \}$ of subspaces of $\cO(\bG)$
(for example, $\{ \cO(\bG)_{\leq d,\phi} \})$,   \ $\phi: M \to N$ \ of $\bG$-modules is an isomorphism if and only 
if  \ $(\phi)_{X_i}: M_{X_i} \to N_{X_i}$ is an isomorphism of $\bG$-modules for all $i$.
\end{prop}

\vskip .1in

Recall the Schur algebra $S(N,d)$, the dual space of the subspace  $\cO(\bM_{N,N};d)$
of $\cO(\bM_{N,,N})$ consisting of polynomials in the matrix coefficients $x_{i,j}$ which 
are homogeneous of degree $d$.
A module for $S(N,d)$ (equivalently, a comodule for $\cO(\bM_{N,N};d)$ ) is called a 
homogeneous polynomial representation of $GL_N$ of degree $d$.

\begin{ex}
\label{ex:Schur}
Let  $M$ be a homogeneous polynomial representation of $GL_N$ of degree $d$.
Then $M_{\cO(GL_N)_{\leq s}}$ equals $0$ if $s < d$ whereas  
$M_{\cO(GL_N)_{\leq s}} = M$ if $s \geq d$.

More generally, let $\phi: \bG \ \subset GL_N$ be a closed embedding of a linear 
algebraic group $\bG$ with the property  that \ 
$A(\bG) \ \equiv \ \cO\bG) \cap \cO(\bM_{N,N})$ can be written as a direct sum 
$\bigoplus_d A(\bG)_d$, where $A(\bG)_d = \cO(\bG) \cap \cO(\bM_{N,N})_d$ (for example, 
the classical orthogonal or symplectic groups).  If $M$ is an object of $CoMod(A(\bG)_d)
\subset Mod(\bG)$, then  $M_{\cO(\bG)_{\leq s,\phi}} = 0$ if $s < d$ and   
$M_{\cO(\bG)_{\leq s, \phi}} = M$ if $s \geq d$.
 See \cite[1.2]{Doty}.
\end{ex}

\vskip .1in

\begin{ex}
Give $\cO(\bG_a) = k[t]$ the evident filtration by degree (equal to that associated to the
embedding of $\phi: \bG_a \hookrightarrow GL_2$ as the unipotent radical of a Borel subgroup).
The subcoalgebra $\cO(\bG_a)_{\leq p^r-1,\phi} \subset \cO(\bG_a)$ is isomorphic 
to the coordinate algebra of $\bG_{a(r)}$; thus, the abelian category $CoMod(\cO(\bG_a)_{\leq p^r-1,\phi})$
is isomorphic to $Mod(\bG_{(r)})$ which in turn is isomorphic to the
category of modules for the elementary abelian $p$-group $(\bZ/p)^{\times r}$; this category 
is wild, if $r > 2$ or if $p > 2, r = 2$.

To give a vector space $M$ the structure of a $\bG_a$-module is equivalent to giving a
sequences of $p$-nilpotent operators $\psi_i: M \to M, i \geq 0$ which pair-wise commute
and which satisfy the condition that for each $m \in M$ there exists some $n_m$ such
that $\psi_i(m) = 0, i \geq n_m$.  Given such a $\bG_a$-module $M$, the $\bG$-submodule
$M_{\cO(\bG_a)_{\leq p^r-1,\phi}} \subset M$ consists of those $m \in M$ such that $\psi_i(m) = 0, 
i \geq r$.
\end{ex}

\vskip .1in

\begin{prop}
\label{prop:tensor-filt}
Let $\bG$ be a linear algebraic group equipped with the closed embedding $\phi: \bG \hookrightarrow GL_N$
and consider two $\bG$-module $M, M^\prime$.  Then 
$$M_{\cO(\bG)_{\leq d,\phi}} \otimes 
M^\prime_{\cO(\bG)_{\leq d^\prime,\phi}} \quad \subset \quad(M\otimes M^\prime)_ {\cO(\bG)_{\leq d+d^\prime,\phi}}.$$
\end{prop}

\begin{proof}
Since multiplication $\mu: \cO(GL_N) \otimes \cO(GL_N) \to \cO(GL_N)$ restricts to \\
$\cO(\bG)_{\leq d,\phi} \otimes \cO(\bG)_{\leq d^\prime,\phi} \ \to \ \cO(\bG)_{\leq d+d^\prime,\phi}$,
the proposition is a consequence of Proposition \ref{prop:prod-coalg}(1).  
\end{proof}

\vskip .1in

The following proposition explains how the ascending converging sequence for a Frobenius twist
$M^{(r)}, \ \{ (M^{(r)})_{\cO(\bG)_{\leq p^r\cdot d}} \}$ \
is determined by $\{ M_{\cO(\bG)_{\leq d}} \}$.

\begin{prop}
\label{prop:Frob-twist}
Assume that $\bG$ is defined over $\bF_{p^r}$ and that $M$ is a $\bG$-module with
coaction $\Delta_M: M \otimes O(\bG)$ also defined over $\bF_{p^r}$.  Then the $r$-th Frobenius twist $M^{(r)}$
of $M$ has coaction 
$$\Delta_{M^{(r)}} \ \simeq \  1_M \otimes (F^r)^*\circ \Delta_M: M \ \to \ M\otimes \cO(\bG) \ \to\ M \otimes \cO(\bG).$$

If $\phi:\bG \hookrightarrow GL_N$ is defined over $\bF_{p^r}$, then 
$$M_{\cO(\bG)_{\leq d}} \quad = \quad (M^{(r)})_{\cO(\bG)_{\leq p^r\cdot d}}.$$
\end{prop}

\begin{proof}
The identification of $\Delta_{M^{(r)}}$ is implicit in \cite[I.9.10]{J} following \cite[\S 1]{FS}.
Granted this, the second assertion follows from the fact that $(F^r)^*: \cO(\bG) \to \cO(\bG)$
sends $f({x_i,j}) \in \cO(GL_N)$ to $f^{p^r}(x_{i,j})$, multiplying the degree of each monomial
by $p^r$.
\end{proof}

\vskip .2in


\section{Mock Injective $\bG$-Modules}
\label{sec:mock}

A $\bG$-module $M$ for a linear algebraic group is called mock injective
if its restriction $M_{|\bG_{(r)}}$ to each Frobenius kernel $\bG_{(r)}$ is an injective $\bG_{(r)}$-module.
By the detection of injectivity property for support varieties for infinitesimal group
schemes (\cite{SFB2}, \cite{Pevt}), $M$ is mock injective if and only if the support variety
$\Pi(\bG_{(r)})_{M_{|\bG_{(r)}}}$ is empty for all $r > 0$.
 
If a mock injective $\bG$-module is not injective (as a $\bG$-module), then it is called
a proper mock injective $\bG$-module.  The existence of proper mock injectives was
first established in \cite{F18}.

The following list of properties of mock injective $\bG$-modules following easily from the 
exactness of $(-)_{\bG_{(r)}}: Mod(\bG) \to Mod(\bG_{(r)})$
and the corresponding properties for support properties for $\bG_{(r)}$-modules.  (See \cite[Prop 4.6]{F18}.)

\begin{prop}
\label{prop:property-mock}
Let $\bG$ be a  linear algebraic group.
\begin{enumerate}
\item
A $\bG$-module is mock injective if and only if its support variety $\Pi(G)_M$ 
(as defined in \cite{F23}) is empty.
\item
A directed colimit $\varinjlim_i M_i$ of mock injective $\bG$-modules is mock injective
\item
Let $0 \to M_1 \to M_2 \to M_3 \to 0$ be an exact sequence of $\bG$-modules.  If
two of $ M_1, M_2, M_3 $ are mock injective, then the third is also mock injective.
\item
If $\mathbb H \hookrightarrow \mathbb G$ is a closed embedding of linear algebraic
groups and $M$ is a mock injective $\bG$-module, then the restriction of
$M$ is a mock injective $\bH$-module.
\item
The tensor product of two mock injective $\bG$-modules is mock injective.
\end{enumerate}
\end{prop}

\vskip .1in

By Proposition  \ref{prop:property-mock}(1), support varieties offer no information 
which might distinguish non-isomorphic mock injectives.   This should be contrasted
to Proposition \ref{prop:functorial} concerning $M \mapsto M_{\cO(\bG)_{\leq d, \phi}}$.

We recall our first construction of proper mock injective $\bG$-modules, an interpretation of
results of Cline Parshall, and Scott concerning induced modules. (See \cite{CPS77}.)

\begin{prop} \cite[Prop 4.54]{F18}
\label{prop:first-mock}
Let $\bG$ be a linear algebraic group and $H \hookrightarrow \bG$ a closed subgroup.
Then the restriction $(k[\bG]^R)_{|H}$ to $H$ of the right regular representation of $\bG$ is
a mock injective $H$-module.  On the other hand, $(k[\bG]^R)_{|H}$ is an injective $H$-module
if and only if $\bG/H$ is an affine variety.
\end{prop}

\vskip .1in

The next construction of proper mock injectives, a summary of results of Hardesty, Nakano, and Sobaje 
in \cite[\S2]{HNS},  involves $\bG$-modules obtained by induction from finite subgroups
$\bG(\bF_p) \hookrightarrow \bG$.

\begin{prop} \cite[\S 2]{HNS}
\label{prop:HNS}
Let $\bG$ be a linear algebraic group $\bG$ defined over $\bF_q$ and $H \hookrightarrow \bG$
a finite subgroup scheme with the property that the Frobenius map $F^q: \bG \to \bG$ restricts
to an automorphism of $H$.  
\begin{enumerate}
\item
The exact functor
$$ind^{\bG}_H(-) \circ (-)_{|H}: (\bG\text{-}modules) \ \to \ (\bG{\text -}modules) ,\quad
M \mapsto ind^{\bG}_H(M_{|H}) \ = \ M\otimes ind^{\bG}_H k$$
takes values in mock injective $\bG$-modules.  
\item
If every simple  $H$-module is the restriction of a $\bG$-module, then $ind^{\bG}_H(M_{|H})$ is an injective
$\bG$-module if and only if $M_{|H}$ is an injective $H$-module.
\end{enumerate}
\end{prop}

\vskip .1in
Applying an induction argument on the length of composition series for $M$, we see that
Proposition \ref{prop:HNS}(2)  provides numerous concrete examples of proper mock injective modules.

\begin{ex}
\label{ex:abundant}
Let $\bG$ be as in Proposition \ref{prop:HNS} and set $H \hookrightarrow \bG$ equal to 
$\bG(\bF_{p^s})$ for some $s \leq d$.   
\begin{enumerate}
\item
Assume that  $\bG$ is semi-simple and that $M$ is a finite dimensional $\bG$-module.  
If the composition series for $M$ 
has no irreducible constituent $S_\lambda$ with $\lambda$ of the form $(p^s-1)\rho + p^s\mu$, 
then $ind^{\bG}_H(M_H)$ is a  proper mock injective $\bG$-module.
\item
If $\bG$ is unipotent and $M$ is a finite dimensional $\bG$-module, then $ind^{\bG}_H(M_{|H}))$
is a proper mock injective $\bG$-module.
\end{enumerate}
\end{ex}

\vskip .2in

\section{Cofinite $\bG$-modules  }
\label{sec:growth}

In this section, we investigate cofinite $\bG$-modules, a class of (necessarily
countable) $\bG$-modules which seem somewhat amenable to study.

\vskip .1in

\begin{defn}
\label{defn:cofinite}
Let $\bG$ be an affine group scheme.   We define a $\bG$-module $M$
to be cofinite if \ $M_X$ \ is finite dimensional for every finite dimensional 
subspace $X \subset \cO(\bG)$.

This condition is equivalent to the condition that each $M_{X_i}$ is a finite dimensional $\bG$-module
for some ascending converging sequence $\{ X_i \}$ of finite dimensional subspaces of $\cO(\bG)$.
\end{defn}

\vskip .1in

We establish various properties of cofinite $\bG$-modules.   We caution the reader 
that $M \otimes k[\bG]^R$ is not cofinite for any infinite dimensional $\bG$-module
Indeed, $(M\otimes k[\bG]^R)_{\cO(\bG)_{\leq 0}}$ can be identified with the 
underlying vector space of $M$ with trivial $\bG$-action.
 
 \begin{prop}
 \label{prop:cofinite}
 Let $\bG$ be an affine group scheme, and let $M, \ E, \ N$ be $\bG$-modules.
\begin{enumerate}
\item 
If $0 \to M \to E \to N \to 0$ is exact and if $M,  \ N$ are cofinite, then $E$ is also cofinite.
\item 
If $\bG = GL_N$ is a linear algebraic group (and more generally if $\bG$ admits an 
embedding \ $\theta: \bG \hookrightarrow \GL_N$ as in Example \ref{ex:Schur}), and if $M$ 
is finite dimensional,
then $M \otimes N$ is cofinite if and only if $N$ is cofinite.
\item 
Any direct summand of a cofinite $\bG$-module is also cofinite.
\item  
If $M$ is finite dimensional, then $M$ embeds in an injective $\bG$-module which is cofinite.
\end{enumerate}  
\end{prop}

\begin{proof}
The left exactness of $(-)_X$ implies assertions (1) and (3).   Assertion (4) is justified by 
the embedding $M \subset M\otimes k[\bG]^R$ together with the observation that
$M\otimes k[\bG]^R$ is isomorphic to $M^{tr}\otimes k[\bG]^R$.

We prove assertion (2) for $\bG = GL_N$; with notational changes, the same proof applies
to $\bG$ admitting \ $\theta: \bG \hookrightarrow \GL_N$ as in Example \ref{ex:Schur}).
Let $e$ be chosen such that $M = M_{\cO(GL_N)_{\leq e}}$.  It suffices to prove
for all $d \geq e$ that 
\begin{equation}
\label{eqn:sub}
(M \otimes N)_{\cO(GL_M)_{\leq d-e}} \ \subset \ M\otimes N_{\cO(GL_N)_{\leq d}}.
\end{equation}
To prove (\ref{eqn:sub}), we proceed by contradiction, assuming that there exists
some $\sum m_{\alpha} \otimes n_{j,\beta} \in (M \otimes N)_{\cO(GL_N)_{\leq d-e}} $ with the 
property that 
\begin{equation}
\label{eqn:sub2}\Delta(\sum m_i \otimes n_j) \ = \ \sum_{\alpha,\beta} m_\alpha \otimes n_\beta \otimes f_{\alpha,i}\cdot f_{\beta,j}
\ \in \ M\otimes N \otimes \cO(GL_N)
\end{equation}
with some $f_{\alpha,i}\cdot f_{\beta,j} \notin \cO(GL_N)_{\leq d}$.
Here,  $\Delta(m_i) = \sum_\alpha m_\alpha \otimes f_{\alpha,i}$
and $\Delta(n_j) = \sum_\beta n_\beta \otimes f_{\beta,j}$, $\{ m_\alpha \}$ is a basis for $M$, 
$\{ n_\beta \}$ is a basis for $N$, and each $f_{\alpha,i}, \ f_{\beta,j}$ is an element of a basis 
$\{ f_\gamma \}$ of $\cO(GL_N)$.  Let $D$ be the largest integer such that some 
$f_{\beta,j} \notin \cO(GL_N)_{\leq D-1}$ and let $g$ denote a choice of such a $f_{\beta,j}$.  Observe that each $f_{\alpha,i}\cdot g
\notin \cO(GL_N)_{\leq d-e}$ since $f_{\alpha,i} \in \cO(GL_N)_{\leq e}$.  To obtain a contradiction, 
it suffices to verify the coefficient of 
$\sum_i m_\alpha \otimes n_\beta \otimes f_{\alpha,i}\cdot g$ in (\ref{eqn:sub2}) is non-zero for some
$m_{\alpha} \otimes n_{\beta} $.  The vanishing of this coefficient is a single linear condition on $\{ f_{\alpha,i}\cdot g\}$.
If this vanishing occurs, it will fail if we replace $n_\beta$ by $c \cdot n_\beta$ and 
$f_{\beta,j}$ by $c^{-1}\cdot f_{\beta,j}$.
\end{proof}

\vskip .in

In the following examples, we see that 
the category \ $CoFin(\bG) \ \subset \ Mod(\bG)$ \ of cofinite $\bG$-modules
is not closed upon quotients.

\begin{ex}
\label{ex:counter}
 Let $\bG$ be a semi-simple linear algebraic group over $k$.  
 Choose a non-trivial extension
 $0 \to M \to E \to k \to 0$ corresponding
 to a non-zero class in $H^1(\bG,M)$ with $M$ finite dimensional.  By 
 Andersen's theorem (see \cite[Thm II.10.16]{J}), the $r$-th Frobenius twist
 of this class in $H^1(\bG,M^{(r)})$ is also non-zero, so that the short exact
 sequence \ $0 \to M^{(r)} \to E^{(r)} \to k \to 0$ is non-split.  Thus, 
 $\bigoplus_{r\geq 0} E_{\lambda,\mu}^{(r)}$ is cofinite, but has as quotient
 $\bigoplus_{r \geq 0} k$ which is not cofinite.
 \end{ex}

\vskip .1in

The following invariant $\gamma(\bG)_M $ of a cofinite $\bG$-module is only one of many similar invariants
one might define.

\begin{defn}
\label{defn:exp-growth}
Let $\bG$ be a linear algebraic group equipped with a closed embedding $\phi: \bG \hookrightarrow GL_N$
and let $M$ a cofinite $\bG$-module.  We say that $M$ has
cofinite type $\gamma(\bG)_M$ equal to $(c,\epsilon)$ if
\begin{equation}
\label{eqn:exp-growth}
 \varinjlim_d \frac{dim(M_{\cO(\bG)_{\leq d\phi}})}{d^\epsilon} \quad = \quad c.
 \end{equation}
 For such a $\bG$-module $M$, we say that $M$ has polynomial growth of degree $\epsilon$.
 
 Thus, in view of (\ref{eqn:advantage3}), Proposition \ref{prop:O(GLN)} tells us that the $\cO(GL_N)$-module 
 $k[GL_N]^R$ has cofinite type $(\frac{1}{(N^2)!},N^2)$ and Example \ref{ex:unipotent}
tells us that the $\cO(\bU_N)$-module $k[\bU_N]^R$ has cofinite type $(\frac{1}{(N(N-1))!},{N(N-1)}$.
\end{defn}

\vskip .1in

The following example include the observation that cofinite type 
differentiates the mock injectives of Propositions \ref{prop:first-mock}.

\begin{ex}
\label{ex:Ga}
Give $\cO(\bG_a) = k[t]$ the evident filtration by degree (equal to that associated to the
embedding of $\phi:\bG_a \hookrightarrow GL_2$ as the unipotent radical of a Borel subgroup).

The cofinite type of the injective $\bG_a$-module $k[t]^L = k[t]^R = k[t]$  equals $(1,1)$,
whereas the cofinite type of the mock injective $\bG_a$-module $ind_{\bG_a(\bF_q)}^{\bG} k$ equals $(1,1/q)$.

On the other hand, the mock injective $\bG_a$-module $(k[GL_N])_{|\bG_a}$ is not cofinite as
a $\bG_a$-module.

In contrast, the $\bG_a$-submodule $P \ \equiv \{1, t^{p^i}\} \subset k[\bG_a]^R = k[t]$ 
of primitive elements
satisfies $dim(P_{\cO(\bG_a)_{\leq p^r}}) \ = \ r+1$ and thus does not have polynomial growth.
\end{ex}

\vskip .1in

\begin{ex}
\label{ex:GL-poly}
Let $P$ be a polynomial representation of $GL_N$ of dimension $n$ which is homogeneous 
of degree $s$ and let $M = S^*(P)$ be the symmetric algebra on $P$ 
viewed as a $GL_N$-module.  Since the coaction of $\cO(GL_N)$ on $P$ factors
through $\cO(\bM_{N,N})$, \ $M$ is a graded $\cO(\bM_{N,N})$-module with 
$M_{\cO(GL_N)_{\leq d\cdot s}}  \ = \ M_{\cO(M_{N,N})_{d\cdot s}}  \ = \ S^d(P)$.
Thus, $dim(M_{\cO(GL_N)_{\leq d\cdot s}}) \ = \ \binom{d+n}{n}$, which as
a function of $d$ 
differs from $d \mapsto \frac{d^d}{s^n\cdot n!}$ by an error term of degree (in $d$)
less than $N$.

Thus, $P$ is a cofinite $GL_N$ module with $\gamma(GL_N)_P = (\frac{1}{s^n\cdot n!},n)$.
\end{ex}

\vskip .1in

As we see below, the polynomial growth of a cofinite $\bG$-module $M$ is independent 
of the choice of closed embedding $\phi:\bG \hookrightarrow GL_N$.

\begin{prop}
\label{prop:equal-growth}
Let $\bG$ be an affine group scheme and $M$ a cofinite $\bG$-module.  Consider
two closed embeddings $\phi: \bG \hookrightarrow GL_N
\ \phi^\prime: \bG \hookrightarrow GL_{N^\prime}$.  If  $\gamma(\bG,\phi)_M \ = \ (\epsilon,c)$ and  
if $\gamma(\bG,\phi^\prime)_M \ = \ (\epsilon^\prime,c^\prime)$,
 then $\epsilon = \epsilon^\prime$.
 \end{prop}
 
 \begin{proof}
 If $\phi, \psi: \bN \to \bN$ are sequences of polynomial growth $e,f$ 
respectively, then $\phi\circ\psi$ has polynomial growth $e\cdot f$.
In particular, given an ascending converging sequence $n \mapsto \phi(n)$
of polynomial growth $e$, then a subsequence $n \mapsto \phi (\psi(n))$ with
$\psi(n)$ growing linearly in $n$ also has growth $e$

Thus, the proposition follows by appealing to Proposition \ref{prop:diff-emb}. 
 \end{proof}

\vskip .1in

We next compute the degree of polynomial growth of $d \mapsto dim(\cO(\bG)_{\leq d,\phi})$.
 
\begin{prop}
\label{prop:growthbG}
Let $\phi: \bG \hookrightarrow GL_N$ be a smooth, closed embedding of affine group schemes, 
and let $\fg$ denote the Lie algebra of $\bG$.  Then the polynomial growth of 
\ $d \mapsto dim(\cO(\bG)_{\leq d,\phi})$ equal to that of $d \mapsto d^{dim(\fg)}$.
\end{prop}

\begin{proof}
As seen in the proof of Proposition \ref{prop:O(GLN)},  $d\mapsto dim(\cO(GL_N)_{\leq d}$ has polynomial
growth of degree $N^2$ and equals that of $d \mapsto \cO(M_{N,N})_{\leq d}$.
Since $\cO(M_{N,N})_{\leq d}$ maps isomorphically to the truncated local ring $\cO(M_{N,N})_{(id)}/\fm_{id}^{d+1}$
which maps onto $\cO(\bG)/\fm_{\bG}^{d+1}$, we conclude that 
the image of $\cO(\bM_{N,N})_{\leq d} \to \cO(\bG)$ has growth 
at least $d\mapsto \cO(\bG)/\fm_{\bG}^{d+1}$ which has polynomial growth equal to 
that of $d \mapsto d^{dim(\fg)}$.

Choose $\ell$ to be the least positive integer such that  $\cO(\bG)_{\leq 1,\phi}
\to \cO(\bG)/m_\bG^\ell$ is injective.  Then $\cO(\bG)_{\leq d,\phi}$ maps injectively to 
$\cO(\bG)/m_\bG^{\ell \cdot d}$, so that the argument of the proof of Proposition 
\ref{prop:equal-growth} implies that the growth of $d\mapsto \cO(\bG)_{\leq d,\phi}$
is not greater than $d \mapsto d^n$.
\end{proof}
 
 \vskip .1in
 
 As a consequence of Proposition \ref{prop:growthbG}, we obtain the following.

\begin{prop}
\label{prop:left-right}
Let $\phi: \bG \hookrightarrow GL_N$ be a smooth, closed embedding of affine group schemes, 
and let $\fg$ denote the Lie algebra of $\bG$.
The $G$-modules $k[\bG]^R, \ k[\bG]^L, \ k[\bG]^{Ad}$\ each have polynomial
growth of degree equal to $dim(\fg)$ with respect to $\{ \cO(\bG)_{\leq d,\phi} \}$.
\end{prop}

\begin{proof}
The growth of $k[\bG]^R$ with respect to $\{ \cC(\bG)_{\leq d,\phi} \}$ equals $dim(\fg)$ by
Proposition \ref{prop:growthbG}.

Consider the anitipode $\sigma: \cO(GL_N) \to \cO(GL_N)$, the map on coordinate algebras induced
by $(-)^{-1}: GL_N \to GL_N$.  For any invertible $N\times N$-matrix $A$, Kramer's rule tell us that $A$
has inverse $B = (b_{i,j})$, where $b_{i,j} = (-1)^{i+j} det(A_{i,j})\cdot det(A)^{-1}$ where
$A_{i,j}$ is the minor of $A$ at $(i,j)$.  Thus, $\sigma(x_{i,j}) \in \cO(GL_N)$ is the product of an $N-1$ degree
polynomial $(-1)^{i+j}det(\{ x_{s,t}, s \not=i, y\not= j \})$ and the function $det(\{x_{i,j})^{-1}$ (which is given filtration degree
$N$).  Consequently, the algebra homomorphism $\sigma: \cO(GL_N)\to \cO(GL_N)$ multiplies filtration 
degree (with respect to $\{ \cO(\GL_N)_{\leq d} \}$) by $2N-1$.
The coaction for $k[GL_N]^L$  is the composition 
$$\tau \circ (\sigma_{GL_N} \otimes id_{\cO(GL_N)}) \circ \Delta_{GL_N}: \cO(GL_N) \ \to \ \cO(GL_N) \otimes \cO(GL_N)
\to \ $$
$$\cO(GL_N) \otimes \cO(GL_N)\ \to \ \cO(GL_N) \otimes \cO(GL_N),$$
where $\tau$ exchanges tensor factors.
Thus, $dim((k[GL_N]^L)_{\cO(GL_N)_{\leq (2N-1)d}} )\ = \ 
dim(\cO(GL_N)_{\leq d})$.  
We conclude that \ $d \mapsto (k[GL_N]^L)_{\cO(GL_N)_{\leq d}}$
has polynomial growth of degree $N^2$ as in the proof of Proposition \ref{prop:equal-growth}.

The coaction determining the comodule structure of $k[GL_N]^{Ad}$ is the composition 
$$\mu_{1,3} \circ (\sigma_{GL_N}\otimes 1 \otimes 1) \circ (1\circ \Delta_{GL_N}) \circ \Delta_{GL_N}:
\cO(GL_N) \ \to \ \cO(GL_N) \otimes \cO(GL_N) \ \to \ $$
where $\mu_{1,3}$ multiplies the first and third tensor factors, then exchanges the two remaining
factors.  (See \cite[I.2.8(70]{J}.)
If we write this coaction as $\Delta_{Ad}(x_{i,j}) =  \sum_{s,t} x_{s,t} \otimes f_{s,t}^{i,j}$,
then we see that each  $f_{s,t}^{i,j}$ is a product of a function of filtration degree $\leq 2N-1$
and a function of degree 1, thus $f_{s,t}^{i,j}$ has filtration degree $\leq 2N$.  As for $k[GL_N]^L$, we
conclude that \ $d \mapsto (k[GL_N]^{Ad})_{\cO(GL_N)_{\leq d}}$
has polynomial growth $N^2$.

Since the restriction $\phi^*: \cO(GL_N) \to \cO(\bG)$ is a surjective map of Hopf algebras 
and since $\cO(\bG)_{\leq d,\phi}$ is defined to be $\phi^*(\cO(\bG)_{\leq d})$, we may 
apply $\phi^*$ to the above arguments
for $GL_N$ to conclude the corresponding statements for $\bG$. 
\end{proof}

\vskip .1in

We conclude our  examples with a simple example of a cofinite $GL_N$-module with growth not bounded
by any polynomial

\begin{ex}
\label{ex:growth-exp}
Let $V$ be the natural representation of $GL_N$ of dimension $N$, a 
polynomial representation homogeneous of degree $1$ (see Example \ref{ex:GL-poly}).  Set $M$ to be 
\ $M \ \equiv \ \bigoplus_{n \geq 0} (V^{(n)})^{\oplus n!}.$ \  Then 
$$dim(M_{\cO(GL_N)_{\leq p^r}} ) \ = \  \sum_{0 \leq s \leq r} N\cdot s! ,$$
so that $d \mapsto dim(M_{\cO(GL_N)_{\leq d}})$ grows faster than any polynomial in $d$. 
\end{ex}

\vskip .2in


\section{General questions}
\label{sec:questions}

We mention a few of the many questions which arise when considering 
filtrations of $\bG$-modules.

\vskip .1in

\begin{question}
\label{ques:base}
To what extent does our discussion of $\bG$-modules for an affine group scheme of 
finite type over a field $k$ of positive characteristic extend to group schemes of finite
type over a (commutative) Noetherian $k$-algebra?
\end{question}

\vskip .1in

\begin{question}
\label{ques:Hopf}
Much of the discussion of this paper applies to Hopf algebras $A$ of countable dimension 
over $k$ which are more general than those of the form $\cO(\bG)$ (see \cite{smontgom}), 
although  a braiding is required at some points.   Does a similar consideration of 
ascending converging sequences of subspaces lead to interesting modules for various 
Hopf algebras which are not co-commutative?
\end{question}

\vskip .1in
\begin{question}
\label{ques:identify-cat}
Let $mod(\bG,X) \ \subset \ mod(\bG)$ denote the abelian category of finite 
dimensional $X$-comodules.  Can one give a useful description of the Balmer
spectrum of the triangulated category $D^b(mod(\bG,X))$ for various subspaces $X \subset \cO(\bG)$?
\end{question}

\vskip .1in
\begin{question}
\label{ques:filt-deg}
Let $\bG$ be a reductive group.  How are the filtrations $\{ \cO_{\pi_n}(\bG) \}$ 
introduced by Jantzen in \cite[App A]{J}
and $\{ \cO(\bG)_{\leq d} \}$ related?
\end{question}

\vskip .1in
\begin{question}
\label{ques:Hoch}
For $\bU$ a unipotent linear algebraic group, can we utilize the ascending converging sequences
of subcoalgebras of $\cO(\bU)$ to enable computations
of the Hochschild cohomology of $\bU$?  (See \cite{F19}.)
\end{question}

\vskip .1in
\begin{question}
\label{ques:realize}
Are there constraints on what positive numbers $d$ can be the degree of polynomial
growth of an indecomposable, cofinite $GL_N$-module with respect to $\{ \cO(GL_N)_{\leq d} \}$?
\end{question}

\vskip .2in


\end{document}